\documentclass[a4paper,10pt,oneside]{article}
\pdfoutput=1
\RequirePackage{amsmath, amsthm, amssymb, mathtools, mathrsfs}
\RequirePackage{tikz} 
\RequirePackage{tikz-cd} 
\RequirePackage{natbib}

\usepackage[margin=1.2in]{geometry}
\usepackage{libertine}
\usepackage{courier}
\usepackage{mathpazo}
\usepackage{xcolor}
\usepackage{pifont}
\usepackage{authblk}
\usepackage{hyperref}
\usepackage{enumerate}

\usepackage{titlesec}

\definecolor{darkbrown}{RGB}{170,100,0} 
\definecolor{darkdarkbrown}{RGB}{110,70,0} 
\titleformat{\section}
{\color{darkdarkbrown}\normalfont\Large\bfseries}
{\color{darkdarkbrown}\thesection}{1em}{}
\titleformat{\subsection}
{\color{darkdarkbrown}\normalfont\Large\bfseries}
{\color{darkdarkbrown}\thesubsection}{1em}{}

\hypersetup{
	colorlinks=true,
  linkcolor=darkbrown,
  citecolor=darkbrown,
  filecolor=darkbrown,
  urlcolor=darkbrown
}

    \makeatletter
\def\@fnsymbol#1{\ensuremath{\ifcase#1\or \dagger\or \ddagger\or
   \mathsection\or \mathparagraph\or \|\or **\or \dagger\dagger
   \or \ddagger\ddagger \else\@ctrerr\fi}}
    \makeatother

\newtheorem{theorem}{Theorem}[section]
\newtheorem{lemma}{Lemma}[section]
\newtheorem{definition}{Definition}[section]

\newtheorem{remark}{Remark}[section]

\providecommand{\keywords}[1]
{
  \small	
  \textbf{\textit{Keywords---}} #1
}
\definecolor{darkgreen}{rgb}{0.2, 0.5, 0.2}

\DeclareMathOperator*{\argmin}{arg\,min}

\newcommand{\pred}[1]{\delta\left[#1\right]}
\newcommand{\set}[1]{\left\{ #1 \right\}}
\newcommand{\eqdef}{\doteq}

\newcommand{\bbR}{\mathbb{R}}
\newcommand{\bbQ}{\mathbb{Q}}

\newcommand{\N}{\mathbb{N}}
\newcommand{\bbN}{\mathbb{N}}

\newcommand{\abs}[1]{\left| #1 \right|}
\newcommand{\nrm}[1]{\left\Vert #1 \right\Vert}
\newcommand{\PR}[2][]{\mathbb{P}_{#1}\left( #2 \right)}

\newcommand{\eps}{\varepsilon}

\newcommand{\bigO}{\mathcal{O}}

\newcommand{\rev}{_{\mathsf{rev}}}
\newcommand{\test}{_{\mathsf{test}}}

\newcommand{\sym}{_{\mathsf{sym}}}

\newcommand{\fshr}{\mathfrak{g}}

\newcommand{\calX}{\mathcal{X}}

\newcommand{\calH}{\mathcal{H}}
\newcommand{\calY}{\mathcal{Y}}

\newcommand{\calP}{\mathcal{P}}
\newcommand{\calE}{\mathcal{E}}

\newcommand{\calW}{\mathcal{W}}
\newcommand{\calQ}{\mathcal{Q}}

\newcommand{\calV}{\mathcal{V}}
\newcommand{\Vtest}{\calV \test}
\newcommand{\calD}{\mathcal{D}}

\newcommand{\calR}{\mathcal{R}}
\newcommand{\calS}{\mathcal{S}}

\newcommand{\rf}[1]{\overline{#1}}

\author[1]{Geoffrey Wolfer \thanks{email: geoffrey.wolfer@riken.jp. \\
The author is supported by the Special Postdoctoral Researcher Program (SPDR) of RIKEN.}}
\author[2]{Shun Watanabe \thanks{email: shunwata@cc.tuat.ac.jp. \\ The author is supported in part by Japan Society for the Promotion of Science KAKENHI under Grant 20H02144.}}
\affil[1]{RIKEN Center for AI Project}
\affil[2]{Department of Computer and Information Sciences \protect\\ Tokyo University of Agriculture and Technology}

\date{\today}

\title{\vspace{-1.5cm} A Geometric Reduction Approach for \\ Identity Testing of Reversible Markov Chains}

\begin{document}
\maketitle

\begin{abstract}
    We consider the problem of testing the identity of a reversible Markov chain against a reference from a single trajectory of observations. Employing the recently introduced notion of a lumping-congruent Markov embedding, we show that, at least in a mildly restricted setting, testing identity to a reversible chain reduces to testing to a symmetric chain over a larger state space and recover state-of-the-art sample complexity for the problem.
\end{abstract}


\section{Introduction}
\label{section:introduction}
Uniformity testing is the flagship problem of distribution property testing. From $n$ independent observations sampled from an unknown distribution $\mu$ on a finite space $\calX$, the goal is to distinguish between the two cases where $\mu$ is uniform and $\mu$ is $\eps$-far from being uniform with respect to some notion of distance.
The complexity of this problem is known to be of the order of $\tilde{\Theta}(\sqrt{\abs{\calX}}/\eps^2)$ \cite{paninski2008coincidence}, which compares favorably with the linear dependency in $\abs{\calX}$ required for estimating the distribution to precision $\eps$ \cite{waggoner2015lp}. 
Interestingly, the uniform distribution can be replaced with some arbitrary reference at same statistical cost.
In fact, Goldreich \cite{goldreich2016uniform} proved that the latter problem formally reduces to the former. Inspired by his approach, we seek and obtain a reduction result in the much less understood and more challenging Markovian setting.

\paragraph{Informal Markovian problem statement ---}

The scientist is given the full description of a reference transition matrix $\rf{P}$ and a single Markov chain (MC) $X_1^n$ sampled with respect to some unknown transition operator $P$ and arbitrary initial distribution.
For fixed proximity parameter $\eps > 0$, the goal is to design an algorithm that distinguishes between the two cases $P = \rf{P}$ and $K(P,\rf{P}) > \eps$, with high probability, where $K$ is a contrast function between stochastic matrices.

\paragraph{Related work ---}

Under the contrast function \eqref{eq:kazakos-contrast} described in Section~\ref{section:preliminaries}, and the hypothesis that $P$ and $\rf{P}$ are both irreducible and symmetric over a finite space $\calX$, \cite{daskalakis2018testing} constructed a tester with sample complexity $\tilde{\bigO}(\abs{\calX}/\eps + h)$, where $h$ \cite[Definition~3]{daskalakis2018testing} corresponds to some hitting property of the chain, and a lower bound in $\Omega(\abs{\calX}/\eps)$.
In \cite{pmlr-v99-cherapanamjeri19a}, a graph partitioning algorithm delivers, under the same symmetry assumption, a testing procedure with sample complexity $\bigO(\abs{\calX}/\eps^4)$, i.e. independent on hitting properties.
More recently, \cite{pmlr-v151-fried22a} relaxed the symmetry requirement,
replacing it with a more natural reversibility assumption.
Their algorithm has a sample complexity of $\bigO(1/(\rf{\pi}_\star \eps^4))$, where $\rf{\pi}_\star$ is the minimum stationary probability of $\rf{P}$, gracefully recovering \cite{pmlr-v99-cherapanamjeri19a} under symmetry.
In parallel, \cite{pmlr-v108-wolfer20a} started the research program of inspecting the problem under the infinity norm for matrices, and derived nearly minimax-optimal bounds.

\paragraph{Contribution ---}

We show how to mostly recover \cite{pmlr-v151-fried22a} under additional assumptions (see Section~\ref{section:restricted-problem}), with a technique based on a geometry preserving embedding. We obtain a more economical proof than \cite{pmlr-v151-fried22a} who went through the process of re-deriving a graph partitioning algorithm for the reversible case. 
Furthermore, our approach, by its generality, is also applicable to related inference problems.

\section{Preliminaries}
\label{section:preliminaries}
We let $\calX, \calY$ be finite sets, and denote $\calP(\calX)$ the set of all probability distributions over $\calX$.
All vectors are written as row vectors.
For matrices $A, B$, $A \circ B$ is their Hadamard product and $\rho(A)$ is the spectral radius of $A$. For $n \in \bbN$, we use the compact notation $x_1^n = (x_1, \dots, x_n)$.
$\calW(\calX)$ is the set of all row-stochastic matrices over the state space $\calX$, and $\pi$ is called a stationary distribution for $P \in \calW(\calX)$ when $\pi P = \pi$.

\paragraph{Irreducibility and reversibility ---} 

We denote $\calW(\calX, \calD)$ the set of irreducible stochastic matrices over a strongly connected digraph $(\calX, \calD)$.
When $P \in \calW(\calX, \calD)$, $\pi$ is unique and we denote $\pi_\star = \min_{x \in \calX} \pi(x) > 0$.
When $P$ verifies the detailed-balance equation $\pi(x)P(x,x') = \pi(x') P(x',x)$ for any $(x,x') \in \calD$,
we say that $P$ is reversible.

\paragraph{Lumpability ---}

In contradistinction with the distribution setting, merging symbols in a Markov chain may break the Markov property. 
For $P \in \calW(\calY, \calE)$ and a surjective map $\kappa \colon \calY \to \calX$, merging elements of $\calY$ together, we say that $P$ is $\kappa$-lumpable \cite{kemeny1983finite} when the resulting process still defines a MC. If so, the resulting transition matrix can be found in \cite[Theorem~6.3.2]{kemeny1983finite}, which we denote as $\kappa_\star P \in \calW(\calX, \kappa_2(\calD))$, with
$$\kappa_2(\calD) \eqdef \set{ (x,x') \in \calX^2 \colon \exists (y,y') \in \calE, (\kappa(y), \kappa(y')) = (x,x') }.$$

\paragraph{Contrast function ---}

We consider the following notion of discrepancy between two stochastic matrices $P, P' \in \calW(\calX)$,
\begin{equation}
    \label{eq:kazakos-contrast}
K(P,P') \eqdef 1 - \rho \left( P^{\circ 1/2} \circ P'^{\circ 1/2} \right).
\end{equation}
Although $K$ first appeared in \cite{daskalakis2018testing} in the context of MC testing, its inception can be traced back to Kazakos \cite{kazakos1978bhattacharyya}.
It is instructive to observe that
$K$ vanishes on chains that share an identical component and does not satisfy the triangle inequality for reducible matrices, hence is not a proper metric on $\calW(\calX)$ \cite[p.10,~footnote~13]{daskalakis2018testing}. Some additional properties of $K$ of possible interest are listed in \cite[Section~7]{pmlr-v151-fried22a}.

\paragraph{Reduction approach for identity testing of distributions ---}

Problem reduction is ubiquitous in the property testing literature.
Our work takes inspiration from \cite{goldreich2016uniform}, who introduced two so-called ``stochastic filters" in order to show how in the distribution setting, identity testing was reducible to uniformity testing, thereby recovering the known complexity of $\bigO(\sqrt{\abs{\calX}}/\eps^2)$ obtained more directly by \cite{valiant2017automatic}. 
Other notable works include
\cite{diakonikolas2016new}, who reduced a collection of distribution testing problems to $\ell_2$-identity testing.

\section{The restricted identity testing problem}
\label{section:restricted-problem}
We cast our problem in the minimax framework by defining the risk $\calR_n(\eps)$,
\begin{equation*}
    \calR_n(\eps) \eqdef \min_{\phi \colon \calX^n \to \set{0, 1}} \set{ \PR[X_1^n \sim \rf{\pi}, \rf{P}]{\phi(X_1^n) = 1} + \max_{P \in \calH_1(\eps)} \PR[X_1^n \sim \pi, P]{\phi(X_1^n) = 0} },
\end{equation*}
sample complexity $n_\star(\eps, \delta) \eqdef \min \set{n \in \N \colon \calR_n(\eps) < \delta}$, and where 
$$\calH_0 = \set{\rf{P}}, \qquad \calH_1(\eps) = \set{ P \in \Vtest \colon K(P, \rf{P}) > \eps},$$ with $\calH_0, \calH_1(\eps) \subset \Vtest$, the subset of stochastic matrices under consideration.
We note the presence of an exclusion region, and that the problem can be regarded as a Bayesian testing problem with a prior which is uniform over $\calH_0$ and $\calH_1(\eps)$ and vanishes on the exclusion region.
We briefly recall the assumptions made in \cite{pmlr-v151-fried22a}. For $(P, \rf{P}) \in (\calH_1(\eps), \calH_0)$,
\begin{enumerate}[$({A}.1)$]
    \item $P$ and $\rf{P}$ are irreducible and reversible.
    \item $P$ and $\rf{P}$ share the same stationary distribution $\rf{\pi} = \pi$.
\footnote{We note that \cite{pmlr-v151-fried22a} also slightly loosen the requirement of having a  matching stationary distributions to being close in the sense where $\nrm{\pi/\rf{\pi} - 1}_\infty < \eps$.}
\end{enumerate}
The following additional assumptions will make our approach readily applicable.
\begin{enumerate}[$({B}.1)$]
    \item $P, \rf{P}$ and share the same connection graph, $P, \rf{P} \in \calW(\calX, \calD)$.
    \item The common stationary probability is rational, $\rf{\pi} \in \bbQ^\calX$.
\end{enumerate}

\begin{remark}
A sufficient condition for $\rf{\pi} \in \bbQ^\calX$ is $\rf{P}(x,x') \in \bbQ$ for any $x,x' \in \calX$.
\end{remark}
Without loss of generality, we express $\rf{\pi} = \left( p_1, p_2, \dots, p_{\abs{\calX}} \right) / \Delta$,
for some $\Delta \in \bbN$, and $p \in \bbN^{\abs{\calX}}$ where $0 < p_1 \leq p_2 \leq \dots \leq p_{\abs{\calX}} < \Delta$.
We denote by $\Vtest$ the set of stochastic matrices that verify assumptions $(A.1), (A.2), (B.1)$ and $(B.2)$ for some fixed and positive $\rf{\pi} \in \calP(\calX)$. 
Our below-stated theorem provides an upper bound on the sample complexity $n_\star(\eps, \delta)$ in $\widetilde{\bigO}(1/(\rf{\pi}_\star \eps))$.

\begin{theorem}
\label{theorem:reduction-theorem}
Let $\eps, \delta \in (0,1)$ and let $\rf{P} \in \Vtest \subset \calW(\calX, \calD)$. There exists a testing procedure $\phi \colon \calX^n \to \set{0,1}$, with $n = \tilde{\bigO}(1/(\rf{\pi}_\star \eps^4))$, such that the following holds. For any $P \in \Vtest$ and $X_1^n$ sampled according to $P$, $\phi$ distinguishes between the cases $P = \rf{P}$ and $K(P, \rf{P}) > \eps$ with probability at least $1 - \delta$.
\end{theorem}

\begin{proof}[sketch]
Our strategy
consists in two steps. First, we employ a transformation on Markov chains termed Markov embedding \cite{wolfer2022geometric} in order to symmetrize both the reference chain (algebraically, by computing the new transition matrix) and the unknown chain (operationally, by simulating an embedded trajectory). Crucially, our transformation preserves the contrast between two chains and their embedded version (Lemma~\ref{lemma:contrast-preservation}).
Second, we invoke the known tester \cite{pmlr-v99-cherapanamjeri19a} for symmetric chains as a black box and report its output.
The proof is deferred to Section~\ref{section:proof-theorem}.
\end{proof}

\begin{remark} Our reduction could also be applied in the robust testing setting \footnote{Note that even in the symmetric setting, the robust problem remains open.}, where the two competing hypotheses are $K(P, \rf{P}) < \eps/2$ and $K(P, \rf{P}) > \eps$.
\end{remark}

\section{Symmetrization of reversible Markov chains}
\label{section:symmetrization}
\paragraph{Information geometry ---}
Our construction and notation follow \cite{nagaoka2005exponential}, who established the dually-flat structure $(\calW(\calX, \calD), \fshr, \nabla^{(e)}, \nabla^{(m)})$ on the space of irreducible stochastic matrices.
Writing $P_\theta \in \calW(\calX, \calD)$ for the transition matrix at coordinates $\theta \in \Theta \subset \bbR^d$, with $d$ the dimension of the manifold, and with the shorthand $\partial_i \eqdef \partial/\partial \theta^i$, recall that the Fisher metric is expressed in the chart induced basis $(\partial_i)_{i \in [d]}$ as
\begin{equation*}
    \fshr_{ij}(\theta) = \sum_{(x,x') \in \calD} \pi_\theta(x) P_\theta(x,x') \partial_i \log P_\theta(x,x') \partial_j \log P_\theta(x,x'), \text{ for } i,j \in [d].
\end{equation*}

\paragraph{Embeddings ---}

In \cite{wolfer2022geometric}, the following notion of an embedding for stochastic matrices is proposed. 

\begin{definition}[Markov embedding for Markov chains \cite{wolfer2022geometric}]
\label{definition:markov-embedding}
We call Markov embedding, a map $\Lambda_\star \colon  \calW(\calX, \calD)  \to \calW(\calY, \calE), P \mapsto \Lambda_\star P$,
such that for any $(y,y') \in \calE$,
$$\Lambda_\star P(y,y') = P(\kappa(y), \kappa(y'))\Lambda(y,y'),$$ and
where $\kappa$ and $\Lambda$ satisfy the following requirements
\begin{enumerate}[$(i)$]
    \item $\kappa \colon \calY \to \calX$ is a lumping function for which $\kappa_2(\calE) = \calD$.
\item $\Lambda$ is a positive function over the edge set, $\Lambda \colon \calE \to \bbR_+$.
\item 
Writing
$\bigcup_{x \in \calX} \calS_x = \calY$ for the partition defined by $\kappa$,
$\Lambda$ is such that for any $y \in \calY$ and $x' \in \calX$, $
    (\kappa(y), x') \in \calD \implies (\Lambda(y,y'))_{y' \in \calS_{x'}} \in \calP(\calS_{x'})$.
\end{enumerate}
\end{definition}
The above embeddings are characterized as the linear maps over lumpable matrices that satisfy some monotonicity requirements and are congruent with respect to the lumping operation \cite[Theorem~3.1]{wolfer2022geometric}.
When for any $y,y' \in \calY$, it additionally holds that $\Lambda(y,y') = \Lambda(y') \pred{ (\kappa(y), \kappa(y')) \in \calD}$, the embedding $\Lambda_\star$ is called memoryless \cite[Section~3.4.2]{wolfer2022geometric} and is e/m-geodesic affine \cite[Th.~3.2, Lemma~3.6]{wolfer2022geometric}, preserving both exponential and mixture families of MC.

Given $\rf{\pi}$ and $\Delta$ as defined in Section~\ref{section:restricted-problem}, from  \cite[Corollary~3.3]{wolfer2022geometric}, there exists a 
lumping function
$\kappa \colon [\Delta] \to \calX$,
and a memoryless embedding $\sigma^{\rf{\pi}}_\star \colon \calW(\calX, \calD) \to \calW([\Delta], \calE)$
with 
$\calE = \set{ (y,y') \in [\Delta]^2 \colon (\kappa(y), \kappa(y')) \in \calD }$,
such that $\sigma^{\rf{\pi}}_\star \rf{P}$ is symmetric. Furthermore, identifying $\calX = \set{1,2, \dots, \abs{\calX}}$, its existence is given constructively by 
$$\kappa(j) = \argmin_{1 \leq i  \leq \abs{\calX}} \set{ \sum_{k=1}^{i} p_k \geq j }, \text{ with } \sigma^{\rf{\pi}}(j) = p^{-1}_{\kappa(j)}, \text{ for any } 1 \leq j \leq \Delta.$$
As a consequence, we have both,
\begin{enumerate}
    \item The expression of $\sigma^{\rf{\pi}}_\star \rf{P}$ following the algebraic manipulations in Definition~\ref{definition:markov-embedding}.
    \item A random algorithm \cite{wolfer2022geometric} to simulate trajectories from $\sigma^{\rf{\pi}}_\star P$ out of trajectories from $P$ (see \cite[Section~3.1]{wolfer2022geometric}).
\end{enumerate}

\section{Contrast preservation}
\label{section:isometry}
It was established in \cite[Lemma~3.1]{wolfer2022geometric} that Markov embeddings preserve the Fisher information metric, the affine connections and the KL divergence between points.
In this section, we show that memoryless embeddings, such as the symmetrizer $\sigma^{\rf{\pi}}_\star$ introduced in Section~\ref{section:symmetrization}, also preserve the contrast function $K$.
Our proof will rely on first showing that memoryless embeddings induce natural Markov morphisms \cite{cencov1978} from distributions over $\calX^n$ to $\calY^n$.

\begin{lemma}
\label{lemma:induced-markov-map-for-path-measure}
Let 
a lumping function $\kappa \colon \calY \to \calX$, and $$L_\star \colon \calW(\calX, \calD) \rightarrow \calW(\calY, \calE)$$ be a $\kappa$ congruent memoryless Markov embedding. For $P \in \calW(\calX, \calD)$, let $Q^n \in \calP(\calX^n)$ (resp. $\widetilde{Q}^n \in \calP(\calY^n)$) be the unique distribution over stationary paths of length $n$ induced from $P$ (resp. $L_\star P$). Then there exists a Markov morphism $M_\star \colon \calP(\calX^n) \to \calP(\calY^n)$ such that $M_\star Q^n = \widetilde{Q}^n$.
\end{lemma}

\begin{proof}

Let $\kappa_n \colon \calY^n \to \calX^n$ be the lumping function on blocks induced from $\kappa$,
\begin{equation*}
    \forall y_1^n \in \calY^n, \kappa_n(y_1^n) = (\kappa(y_t))_{1 \leq t \leq n} \in \calX^n,
\end{equation*}
and introduce 
\begin{equation*}
\calY^n = \bigcup_{x_1^n \in \calX^n} \calS_{x_1^n}, \text{ with } \calS_{x_1^n} = \set{ y_1^n \in \calY^n \colon \kappa_n(y_1^n) = x_1^n },
\end{equation*}
the partition associated to $\kappa_n$.
For any realizable path $x_1^n, Q^n(x_1^n) > 0$, we define a distribution $M^{x_1^n} \in \calP(\calY^n)$ concentrated on $\calS_{x_1^n}$, and such that for any $y_1^n \in \calS_{x_1^n}$, $
    M^{x_1^n}(y_1^n) = \prod_{t=1}^{n} L(y_t).$
Non-negativity of $M^{x_1^n}$ is immediate, and
\begin{equation*}
\begin{split}
    \sum_{y_1^n \in \calY^n} M^{x_1^n}(y_1^n) &= \sum_{y_1^n \in \calY^n \colon \kappa_n(y_1^n) = x_1^n} M^{x_1^n}(y_1^n) %
    = \prod_{t=1}^{n} \left( \sum_{y_t \in \calS_{x_t}} L(y_t) \right)  = 1,
\end{split}
\end{equation*}
thus $M^{x_1^n}$ is well-defined.
Furthermore, for $y_1^n \in \calY^n$, it holds that
\begin{equation*}
\begin{split}    
\widetilde{Q}^n(y_1^n) &= L_\star \pi(y_1) \prod_{t=1}^{n-1} L_\star P(y_t, y_{t+1}) \stackrel{(\spadesuit)}{=} \pi(\kappa(y_1)) L(y_1) \prod_{t=1}^{n-1} P(\kappa(y_t), \kappa(y_{t+1})) L(y_t) \\
&= Q^n(\kappa(y_1), \dots, \kappa(y_n)) \prod_{t=1}^{n} L(y_t) = Q^n(\kappa_n(y_1^n)) \prod_{t=1}^{n} L(y_t) \\
&= \sum_{x_1^n \in \calX^n} Q^n(\kappa_n(y_1^n)) M^{x_1^n}(y_1^n) = M_\star Q^n(y_1^n),
\end{split}
\end{equation*}
where $(\spadesuit)$ stems from \cite[Lemma~3.5]{wolfer2022geometric}, whence our claim holds.
\end{proof}

Lemma~\ref{lemma:induced-markov-map-for-path-measure} essentially states that
the following diagram commutes
\[ \begin{tikzcd}
  \calW(\calX, \calD) \arrow{r}{L_\star} \arrow[swap]{d}{} & L_\star \calW(\calX, \calD)  \arrow{d}{} \\%
\calQ^n_{\calW(\calX, \calD)} \arrow{r}{M_\star}& \calQ^n_{L_\star \calW(\calX, \calD)},
\end{tikzcd}
\]
for some Markov morphism $M_\star$,
and where we denoted $\calQ^n_{\calW(\calX, \calD)} \subset \calP(\calX^n)$ for the set of all distributions over paths of length $n$ induced from the family $\calW(\calX, \calD)$.
As a consequence, we can unambiguously write $L_\star Q^n \in \calQ^n_{L_\star \calW(\calX, \calD)}$ for the distribution over stationary paths of length $n$ that pertains to $L_\star P$.

\begin{lemma}
\label{lemma:contrast-preservation}
Let $L_\star \colon \calW(\calX, \calD) \to \calW(\calY, \calE)$ be a memoryless embedding,
\begin{equation*}
    K(L_\star P, L_\star \rf{P}) = K( P, \rf{P}).
\end{equation*}
\end{lemma}

\begin{proof}

We recall for two distributions $\mu, \nu \in \calP(\calX)$ the definition of $R_{1/2}$ the R\'{e}nyi entropy of order $1/2$,
\begin{equation*}
\begin{split}
    R_{1/2}(\mu \| \nu) \eqdef -2 \log \left( \sum_{x \in \calX} \sqrt{\mu(x) \nu(x)} \right), 
\end{split}
\end{equation*}
and note that $R_{1/2}$ is closely related to the Hellinger distance between $\mu$ and $\nu$.
This definition extends to a divergence rate between stochastic processes $(X_t)_{t \in \N}, (X'_t)_{t \in \N}$ on $\calX$ as follows
\begin{equation*}
    R_{1/2}\left((X_t)_{t \in \N} \| (X'_t)_{t \in \N}\right) = \lim_{n \to \infty} \frac{1}{n} R_{1/2}\left(X_1^n \| X_1'^n \right),
\end{equation*}
and in the irreducible time-homogeneous Markovian setting where $(X_t)_{t \in \N}, (X'_t)_{t \in \N}$ evolve according to transition matrices $P$ and $P'$, the above reduces \cite{rached2001renyi} to 
\begin{equation*}
    R_{1/2}\left((X_t)_{t \in \N} \| (X'_t)_{t \in \N}\right) = -2 \log \rho (P^{\circ 1/2} \circ P'^{\circ 1/2}) = -2 \log (1 - K(P, P')).
\end{equation*}
Reorganizing terms and plugging-for the embedded stochastic matrices,
\begin{equation*}
\begin{split}
    K(L_\star P, L_\star \rf{P}) &=  1 - \exp \left(-\frac{1}{2} \lim_{n \to \infty} \frac{1}{n} R_{1/2}\left( L_\star Q^n \| L_\star \rf{Q}^n \right) \right), \\
\end{split}
\end{equation*}
where $L_\star \rf{Q}^n$ is the distribution over stationary paths of length $n$ induced by the embedded $L_\star \rf{P}$. For any $n \in \N$, from Lemma~\ref{lemma:induced-markov-map-for-path-measure} and information monotonicity of the R\'{e}nyi divergence,
$R_{1/2}\left( L_\star Q^n \| L_\star \rf{Q}^n \right)
    = R_{1/2}\left(  Q^n \| \rf{Q}^n \right),$
hence our claim.
\end{proof}

\section{Proof of Theorem~\ref{theorem:reduction-theorem}}
\label{section:proof-theorem}

We assume that $P$ and $\rf{P}$ are in $\Vtest$.
We reduce the problem as follows.
We construct $\sigma^{\rf{\pi}}_\star$, the symmetrizer \footnote{If we wish to test for the identity of multiple chains against a same reference, we only need to perform this step once.} defined in Section~\ref{section:symmetrization}. 
We proceed to embed both the reference chain (using Definition~\ref{definition:markov-embedding}) and
and the unknown trajectory (using the operational definition in \cite[Section~3.1]{wolfer2022geometric}). We invoke the  tester of \cite{pmlr-v99-cherapanamjeri19a} as a black box, and report its answer.

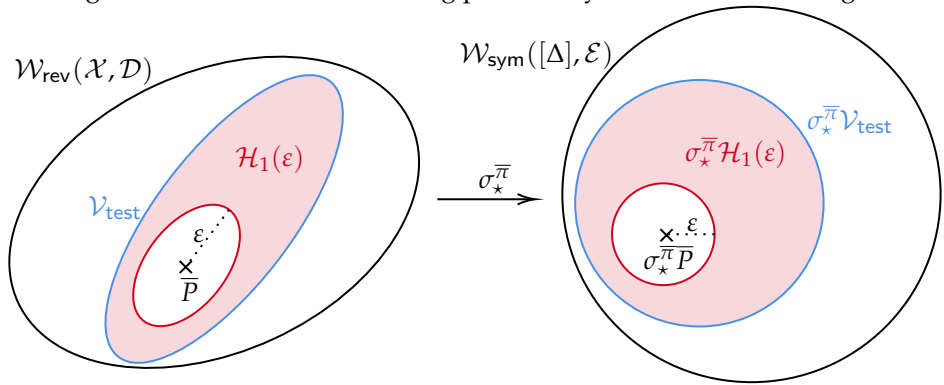
\begin{figure}
\caption{Reduction of the testing problem by isometric embedding.}
\centering

\tikzset{every picture/.style={line width=0.75pt}} %

\begin{tikzpicture}[x=0.60pt,y=0.60pt,yscale=-1,xscale=1]

\draw  [color={rgb, 255:red, 74; green, 144; blue, 226 }  ,draw opacity=1 ][fill={rgb, 255:red, 208; green, 2; blue, 27 }  ,fill opacity=0.15 ] (109.11,245.12) .. controls (90.42,230.92) and (104.76,180.6) .. (141.14,132.71) .. controls (177.51,84.82) and (222.15,57.51) .. (240.83,71.7) .. controls (259.52,85.89) and (245.18,136.22) .. (208.8,184.11) .. controls (172.43,232) and (127.79,259.31) .. (109.11,245.12) -- cycle ;
\draw    (308,146) -- (367,146) ;
\draw [shift={(369,146)}, rotate = 180] [color={rgb, 255:red, 0; green, 0; blue, 0 }  ][line width=0.75]    (10.93,-3.29) .. controls (6.95,-1.4) and (3.31,-0.3) .. (0,0) .. controls (3.31,0.3) and (6.95,1.4) .. (10.93,3.29)   ;
\draw  [color={rgb, 255:red, 74; green, 144; blue, 226 }  ,draw opacity=1 ][fill={rgb, 255:red, 208; green, 2; blue, 27 }  ,fill opacity=0.15 ] (394,148.5) .. controls (394,105.7) and (428.7,71) .. (471.5,71) .. controls (514.3,71) and (549,105.7) .. (549,148.5) .. controls (549,191.3) and (514.3,226) .. (471.5,226) .. controls (428.7,226) and (394,191.3) .. (394,148.5) -- cycle ;
\draw  [color={rgb, 255:red, 208; green, 2; blue, 27 }  ,draw opacity=1 ][fill={rgb, 255:red, 255; green, 255; blue, 255 }  ,fill opacity=1 ] (417,168) .. controls (417,150.33) and (431.33,136) .. (449,136) .. controls (466.67,136) and (481,150.33) .. (481,168) .. controls (481,185.67) and (466.67,200) .. (449,200) .. controls (431.33,200) and (417,185.67) .. (417,168) -- cycle ;
\draw  [color={rgb, 255:red, 208; green, 2; blue, 27 }  ,draw opacity=1 ][fill={rgb, 255:red, 255; green, 255; blue, 255 }  ,fill opacity=1 ] (124.93,222.62) .. controls (113.88,214.23) and (116.83,191.74) .. (131.52,172.4) .. controls (146.21,153.07) and (167.07,144.19) .. (178.13,152.59) .. controls (189.18,160.98) and (186.23,183.46) .. (171.54,202.8) .. controls (156.85,222.14) and (135.98,231.01) .. (124.93,222.62) -- cycle ;
\draw  [dash pattern={on 0.84pt off 2.51pt}]  (178.13,152.59) -- (151.53,187.6) ;
\draw  [dash pattern={on 0.84pt off 2.51pt}]  (481,168) -- (449,168) ;
\draw   (46.45,211) .. controls (26.02,167.01) and (64.22,105.92) .. (131.78,74.55) .. controls (199.33,43.17) and (270.66,53.4) .. (291.09,97.39) .. controls (311.51,141.38) and (273.31,202.47) .. (205.75,233.85) .. controls (138.2,265.22) and (66.87,254.99) .. (46.45,211) -- cycle ;
\draw  [color={rgb, 255:red, 0; green, 0; blue, 0 }  ,draw opacity=1 ] (386,143) .. controls (386,77.83) and (438.83,25) .. (504,25) .. controls (569.17,25) and (622,77.83) .. (622,143) .. controls (622,208.17) and (569.17,261) .. (504,261) .. controls (438.83,261) and (386,208.17) .. (386,143) -- cycle ;
\draw   (147.2,184.8) -- (154.8,192.41)(154.8,184.8) -- (147.2,192.41) ;
\draw   (447.2,164.8) -- (454.8,172.41)(454.8,164.8) -- (447.2,172.41) ;

\draw (332,122) node [anchor=north west][inner sep=0.75pt]   [align=left] {$\sigma_\star^{\rf{\pi}}$};
\draw (42,55) node [anchor=north west][inner sep=0.75pt]  [color={rgb, 255:red, 0; green, 0; blue, 0 }  ,opacity=1 ] [align=left] {$\calW \rev(\calX, \calD)$};
\draw (320,45) node [anchor=north west][inner sep=0.75pt]   [align=left] {$\calW \sym([\Delta], \calE)$};
\draw (154,163) node [anchor=north west][inner sep=0.75pt]   [align=left] {$\eps$};
\draw (462,156) node [anchor=north west][inner sep=0.75pt]   [align=left] {$\eps$};
\draw (146,193) node [anchor=north west][inner sep=0.75pt]   [align=left] {$\rf{P}$};
\draw (435,173) node [anchor=north west][inner sep=0.75pt]   [align=left] {$\sigma^{\rf{\pi}}_\star \rf{P}$};
\draw (180,110) node [anchor=north west][inner sep=0.75pt]  [color={rgb, 255:red, 208; green, 2; blue, 27 }  ,opacity=1 ] [align=left] {$\calH_1(\eps)$};
\draw (461,105) node [anchor=north west][inner sep=0.75pt]  [color={rgb, 255:red, 208; green, 2; blue, 27 }  ,opacity=1 ] [align=left] {$\sigma_\star^{\rf{\pi}} \calH_1(\eps)$};
\draw (88,142) node [anchor=north west][inner sep=0.75pt]  [color={rgb, 255:red, 74; green, 144; blue, 226 }  ,opacity=1 ] [align=left] {$\Vtest$};
\draw (537,85) node [anchor=north west][inner sep=0.75pt]  [color={rgb, 255:red, 74; green, 144; blue, 226 }  ,opacity=1 ] [align=left] {$\sigma_\star^{\rf{\pi}} \Vtest$};

\end{tikzpicture}
\end{figure}

\paragraph{Completeness case.} It is immediate that $P = \rf{P} \implies L_\star P = L_\star \rf{P}$.

\paragraph{Soundness case.}

From Lemma~\ref{lemma:contrast-preservation}, $K(P, \rf{P}) > \eps \implies K(\sigma^{\rf{\pi}}P, \sigma^{\rf{\pi}}\rf{P}) > \eps$.

As a consequence of \cite[Theorem~10]{pmlr-v99-cherapanamjeri19a}, the sample complexity of testing is upper bounded by $\bigO(\Delta/\eps^4)$.
With $\rf{\pi}_\star = p_1/\Delta$ and treating $p_1$ as a small constant, we recover the known sample complexity.

\bibliography{bibliography}
\bibliographystyle{abbrvnat}
\end{document}